\documentclass[11pt]{article}

\usepackage[a4paper,margin=1.15in]{geometry}
\usepackage{amsmath,amssymb,amsthm,mathtools}
\usepackage{booktabs}
\usepackage{microtype}
\usepackage[numbers,sort&compress]{natbib}
\usepackage{xcolor}
\usepackage{hyperref}

\hypersetup{
  colorlinks=true,
  linkcolor=blue!55!black,
  citecolor=green!40!black,
  urlcolor=blue!55!black,
  pdfauthor={Samuel Korsky},
  pdftitle={Arithmetic Progression-Free Subset-Sum Sets}
}
\allowdisplaybreaks
\numberwithin{equation}{section}

\newtheorem{theorem}{Theorem}[section]
\newtheorem{lemma}[theorem]{Lemma}
\newtheorem{proposition}[theorem]{Proposition}
\newtheorem{corollary}[theorem]{Corollary}

\theoremstyle{definition}

\theoremstyle{remark}
\newtheorem{remark}[theorem]{Remark}

\newcommand{\N}{\mathbb{N}}
\newcommand{\Z}{\mathbb{Z}}
\newcommand{\F}{\mathbb{F}}
\newcommand{\eps}{\varepsilon}
\newcommand{\bw}{\operatorname{bw}}
\newcommand{\ceil}[1]{\left\lceil #1\right\rceil}
\newcommand{\floor}[1]{\left\lfloor #1\right\rfloor}

\title{\textbf{Arithmetic Progression-Free Subset-Sum Sets}}
\author{Samuel Korsky}
\date{June 22, 2026}

\begin{document}

\maketitle

\begin{abstract}
\noindent
For a finite set $A$ of positive integers, let $H(A)$ be its set of subset sums, including the empty sum, and let $g_k(n)$ be the least $N$ for which some $n$-element set $A\subseteq[N]$ has $H(A)$ free of nonconstant $k$-term arithmetic progressions. The problem of determining $g_k(n)$ was posed by Erd\H{o}s and S\'ark\H{o}zy. In the three-term case, we prove a lower bound equal to the exact bandwidth of the ternary grid. If $T_m=[x^m](1+x+x^2)^m$ is the central trinomial coefficient, then
\[
g_3(n)\ge \frac{T_n-1}{2}+\sum_{j=0}^{n-1}T_j
 =\left(\frac{\sqrt{3}}{2\sqrt{\pi}}+o(1)\right)\frac{3^n}{\sqrt{n}}.
\]
For general $k \ge 4$ we show
\[
g_k(n)\gg_k \left(\frac{k-1}{k-2}\right)^n
n^{-\log_2((k-1)/(k-2))}
\]
In the opposite direction, a carry-free digit construction based on nearly-regular graphs gives
\[
\limsup_{n\to\infty}g_k(n)^{1/n}
 \le \min_{p\ \mathrm{prime},\ p\ge3}p^{2/(\min\{p,k\}-1)}.
\]
Consequently, as $k\to\infty$, the logarithm of the lower exponential rate is at least $(1+o(1))/k$, while the logarithm of the upper exponential rate is at most $(2+o(1))\log k/k$.
\end{abstract}

\section{Introduction}

For a positive integer $N$, write $[N]=\{1,\ldots,N\}$. If $A=\{a_1,\ldots,a_n\}$ is a finite set of positive integers, its subset-sum set is
\begin{equation}\label{eq:subset-sum-set}
H(A)=\left\{\sum_{i=1}^n\eps_i a_i:\eps_i\in\{0,1\}\right\}.
\end{equation}
We include the empty sum $0$. A $k$-term arithmetic progression is a set of the form $\{x,x+d,\ldots,x+(k-1)d\}$, and it is nonconstant when $d\ne0$. Following Erd\H{o}s and S\'ark\H{o}zy, define $g_k(n)$ to be the least $N$ such that there is an $n$-element set $A\subseteq[N]$ for which $H(A)$ contains no nonconstant $k$-term arithmetic progression. All unadorned logarithms are natural, and constants implicit in $\ll_k$, $\gg_k$, and $O_k(\cdot)$ may depend on $k$.

\bigskip
\noindent
The problem asks how economically one can place $n$ distinct positive generators while forcing all of their subset sums to avoid a fixed progression length. There are two competing mechanisms. If the subset sums collide heavily, then $H(A)$ can be small enough to fit into a short interval, but additive collisions tend to create progressions. At the other extreme, digit constructions make progression-freeness transparent but usually spread the generators over an exponentially long interval. The function $g_k(n)$ measures the optimal compromise.

\bigskip
\noindent
Erd\H{o}s and S\'ark\H{o}zy proved that $g_3(n)\gg3^n/n^{O(1)}$ and asked in particular whether $g_3(n)\gg3^n$; the problem remains open in that form \cite{ErdosSarkozy1992,Bloom817}. Our first result sharpens the known polynomial factor and gives an exact finite lower bound.

\bigskip
\noindent
For $m\ge0$, let
\begin{equation}\label{eq:central-trinomial}
T_m=[x^m](1+x+x^2)^m
\end{equation}
be the $m$th central trinomial coefficient.

\begin{theorem}\label{thm:main-k3}
For every $n\ge1$,
\begin{equation}\label{eq:main-k3-finite}
g_3(n)\ge b_n:=\frac{T_n-1}{2}+\sum_{j=0}^{n-1}T_j.
\end{equation}
Consequently,
\begin{equation}\label{eq:main-k3-asymptotic}
g_3(n)\ge \left(\frac{\sqrt{3}}{2\sqrt{\pi}}+o(1)\right)\frac{3^n}{\sqrt{n}}.
\end{equation}
\end{theorem}

\smallskip
\noindent
The proof has two conceptual steps. First, $H(A)$ is three-term progression-free if and only if all $3^n$ sums $\sum_i\eps_i a_i$ with $\eps_i\in\{0,1,2\}$ are distinct. Thus $g_3(n)$ is an integer-linear layout problem on the ternary grid $\{0, 1, 2\}^n$. Second, the values of that linear form order the vertices of $\{0,1,2\}^n$, and each grid edge has numerical length at most $\max A$. The exact bandwidth formula of Billera and Blanco \cite{BilleraBlanco2013} then gives \eqref{eq:main-k3-finite}. 

\bigskip
\noindent
For general $k$, we use a growth argument for partial subset-sum sets. Suppose some generators have already been chosen, and let $B$ be the set of subset sums they generate. Adding a new generator $h$ replaces $B$ by $B\cup(B+h)$.  If this union is $k$-term-progression-free, then $B$ cannot contain an $h$-chain (an arithmetic progression with common difference $h$) of length $k-1$.  Hence the elements of $B$ split into $h$-chains of length at most $k-2$, and adding $h$ creates at least one new point in each chain. This gives the universal one-step expansion
\[
|B\cup(B+h)|\ge \frac{k-1}{k-2}\cdot|B|.
\]
Additionally, if $U$ is the set of unused generators, then
\[
\sum_{h\in U}|B\cap(B+h)|\le \binom{|B|}{2},
\]
because each unordered pair of elements of $B$ has at most one positive difference. Thus some unused generator has unusually small overlap with $B$. Choosing the next generator adaptively combines this averaged overlap estimate with the universal chain expansion.

\bigskip
\noindent
Set
\begin{equation}\label{eq:dk-lambda}
d_k=\frac{k-1}{k-2},\qquad \lambda_k=\log_2d_k.
\end{equation}

\begin{theorem}\label{thm:main-lower-general}
For every fixed $k\ge4$,
\begin{equation}\label{eq:main-general-lower}
g_k(n)\gg_k d_k^n n^{-\lambda_k}.
\end{equation}
In particular,
\begin{equation}\label{eq:liminf-general-lower}
\liminf_{n\to\infty}g_k(n)^{1/n}\ge\frac{k-1}{k-2}.
\end{equation}
\end{theorem}

\smallskip
\noindent
The exponential base $d_k$ comes from the one-shift chain argument. The distinctness and positivity of the generators enter only in the averaging step; they improve the polynomial factor in the final lower bound. By contrast, the earlier cube-growth lemma of Dietmann and Elsholtz, which also allows repeated directions, gives the weaker exponential base $k/(k-1)$ \cite{DietmannElsholtz2015}. A precise integer-valued version of the adaptive argument is given in Theorem~\ref{thm:adaptive-recurrence} and Corollary~\ref{cor:exact-recurrence-bound}.

\bigskip
\noindent
Our upper bound is a distinct-generator version of a restricted-digit construction. If $p$ is prime and the allowed base-$p$ digits are $0,1,\ldots,\min\{p,k\}-2$, then the resulting digit language is $k$-term-progression-free. Repeated copies of powers of $p$ would realize this language directly, but repetitions are forbidden because $A$ must be a set. We replace repeated generators by two-coordinate generators indexed by the edges of a nearly regular graph.

\bigskip
\noindent
For a prime $p\ge3$, put
\begin{equation}\label{eq:q-rho-definition}
q_{p,k}=\min\{p,k\}-1,
\qquad
\rho_{p,k}(n)=\max\left\{q_{p,k}-1,\ceil{\frac{2n}{q_{p,k}}}\right\}.
\end{equation}

\begin{theorem}\label{thm:main-upper-general}
Let $k\ge3$ and let $p\ge3$ be prime. Then, for every $n\ge1$,
\begin{equation}\label{eq:finite-upper-general}
g_k(n)<2p^{\rho_{p,k}(n)-1}.
\end{equation}
Consequently,
\begin{equation}\label{eq:limsup-upper-general}
\limsup_{n\to\infty}g_k(n)^{1/n}
 \le U_k:=\min_{\substack{p\ge3\\p\ \mathrm{prime}}}p^{2/(\min\{p,k\}-1)}.
\end{equation}
\end{theorem}

\smallskip
\noindent
Choosing a prime $p=(1+o(1))k$ gives the following for large $k$:

\begin{corollary}\label{cor:large-k-rates}
As $k\to\infty$,
\begin{equation}\label{eq:large-k-rates}
\frac{1+o(1)}{k}
 \le \log\!\left(\liminf_{n\to\infty}g_k(n)^{1/n}\right)
 \le \log\!\left(\limsup_{n\to\infty}g_k(n)^{1/n}\right)
 \le (2+o(1))\cdot\frac{\log k}{k}.
\end{equation}
\end{corollary}

\smallskip
\noindent
The missing factor of $\log k$ in the lower bound is the principal gap for large $k$. 

\bigskip
\noindent
We are not aware of previous statements of the finite bound in Theorem~\ref{thm:main-k3}, the adaptive recurrence in Theorem~\ref{thm:adaptive-recurrence}, the universal chain recurrence in Corollary~\ref{cor:chain-only-recurrence}, or the distinct-generator graph construction in Theorem~\ref{thm:graph-construction}. The one-shift ingredients are elementary and the surrounding literature includes work under the terminology of Hilbert cubes, bounded-coefficient dissociated sets, detecting sets, and detecting matrices.

\section{Related Work}\label{sec:related}

\subsection{Arithmetic Progressions in Subset Sums}

The systematic study of arithmetic progressions in subset-sum sets goes back at least to Erd\H{o}s and S\'ark\H{o}zy \cite{ErdosSarkozy1992}. Their work contains the original form of the present problem and proves a lower bound of the shape $3^n/n^{O(1)}$ for $g_3(n)$. The current status and Erd\H{o}s's question $g_3(n)\gg3^n$ are also recorded as Erd\H{o}s Problem 817 \cite{Bloom817}.

\bigskip
\noindent
A complementary line of work asks for long progressions under density or sparsity hypotheses on the generators. Schoen proved that if $A\subseteq[N]$, then the longest arithmetic progression in $H(A)$ has length at least a constant multiple of $|A|/\log N$ \cite{Schoen2011}. This already implies an exponential lower bound of the form $\log g_k(n)\gg n/k$, but it does not produce the factor $\log k$ suggested by restricted-digit examples. Szemer\'edi and Vu established much stronger structural results when the generator set is sufficiently dense in its ambient interval, including arithmetic progressions and proper generalized arithmetic progressions in the subset-sum set \cite{SzemerediVu2006}. Conlon, Fox, and Pham subsequently obtained homogeneous versions of the generalized-arithmetic-progression conclusions \cite{ConlonFoxPham2023}. These density and structure theorems are powerful in polynomial regimes, but their fixed-parameter hypotheses and constants do not presently resolve the exponentially sparse regime relevant to $g_k(n)$.

\bigskip
\noindent
Quantitative forms of Szemer\'edi's theorem can be applied after one has a lower bound for $|H(A)|$. If $r_k(X)$ denotes the largest size of a $k$-term-progression-free subset of an interval of length $X$, then any cube-size lower bound can be inverted through $r_k$. The strongest general estimates currently available include the polylogarithmic bound for $r_4$ of Green and Tao \cite{GreenTao2017} and the subexponential density saving for fixed $k\ge5$ of Leng, Sah, and Sawhney \cite{LengSahSawhney2024}. In our setting these estimates improve lower-order factors but do not alter the exponential base obtained from the adaptive recurrence.

\subsection{Hilbert Cubes and Sumset Growth}

A set of the form $a_0+\{0,a_1\}+\cdots+\{0,a_d\}$ is usually called a Hilbert cube. Dietmann and Elsholtz developed sumset-growth methods for Hilbert cubes contained in progression-free and other arithmetic sets \cite{DietmannElsholtz2012,DietmannElsholtz2015}. Their general cube lemma states that a $d$-dimensional Hilbert cube contained in a set with no $k$-term arithmetic progression has at least $2(k/(k-1))^{d-1}-1$ elements. Their definition allows repeated directions.

\bigskip
\noindent
Lemma~\ref{lem:chain-overlap} below gives the stronger one-step factor $(k-1)/(k-2)$ without requiring directions to be positive or distinct; only nonzero directions are needed. The pairwise distinct positive hypothesis enters later, in Lemma~\ref{lem:average-overlap}, where all remaining direction differences are averaged simultaneously. Related recent work on growth of progression-free sumsets includes Elsholtz, Ruzsa, and Wurzinger \cite{ElsholtzRuzsaWurzinger2025}. Recent general frameworks for Hilbert cubes in arithmetic sets have also been developed by Croot, Mao, and Yip \cite{CrootMaoYip2026}, although their principal applications concern multiplicatively defined ambient sets rather than the extremal function studied here.

\subsection{Distinct Subset Sums and Grid Bandwidth}

A set is subset-sum-distinct, or dissociated, if all of its ordinary subset sums are distinct. Bae introduced $q$-fold subset-sum-distinct sets, for which all coefficient sums with coefficients in $\{0,1,\ldots,q\}$ are distinct \cite{Bae2002}. In the three-term case of the present problem, the relevant notion is exactly two-fold subset-sum-distinctness. Closely related objects also occur under the names bounded-coefficient dissociated sets, detecting sets, and detecting sequences or matrices. Recent work of Dutta studies greedy algorithms and quantitative bounds for dissociated sets and their generalizations \cite{Dutta2026}. Specializing Dutta's general $D_q$-set estimate to $q=2$ gives a lower bound with leading term $\frac{\sqrt3}{4\sqrt\pi}\cdot3^n/\sqrt n$ for the largest element of a two-fold subset-sum-distinct $n$-set. The bandwidth computation in Theorem~\ref{thm:main-k3} doubles this leading constant and supplies the finite expression in \eqref{eq:main-k3-finite}.

\bigskip
\noindent
The graph-theoretic ingredient in Theorem~\ref{thm:main-k3} belongs to the classical bandwidth problem. Harper proved optimality of the Hales order for hypercubes \cite{Harper1966}, and Moghadam extended the optimality statement to Cartesian products of paths \cite{Moghadam2005}. Billera and Blanco obtained a numerical formula for the bandwidth of equal path products and, in particular, an exact formula for the ternary grid used below \cite{BilleraBlanco2013}. The connection between this bandwidth and three-term-progression-free subset sums is especially rigid: it gives the lower bound in Theorem~\ref{thm:main-k3} and precisely identifies the information discarded by the graph-layout relaxation.

\subsection{Restricted Digits}

Restricted-digit sets are a standard source of progression-free examples. For prime base $p$, omitting at least one residue from the allowed digits prevents a progression of length $p$, and smaller digit alphabets prevent shorter progressions as well. When the progression length $k$ is prime, Dietmann and Elsholtz use base $k$, digits $0,1,\ldots,k-2$, and repeated powers of $k$ to demonstrate near-sharpness for general Hilbert cubes \cite{DietmannElsholtz2015}. The primality condition is essential for that formulation: in base $4$, the numbers $0,2,4,6$ use only the digits $0,1,2$ but form a four-term arithmetic progression. Repeated powers cannot be used directly in the definition of $g_k(n)$. The construction in Section~\ref{sec:upper} instead uses a prime base and resolves the distinctness obstruction by replacing repetitions with graph edges while retaining a carry-free digit description.

\section{Preliminaries}\label{sec:preliminaries}

Throughout the paper, $A$ is a finite set of distinct positive integers. If $A\subseteq[N]$ and $|A|=n$, then
\begin{equation}\label{eq:max-sum}
\sum_{a\in A}a\le N+(N-1)+\cdots+(N-n+1)=nN-\binom{n}{2}.
\end{equation}
Consequently,
\begin{equation}\label{eq:interval-upper-H}
|H(A)|\le nN-\binom{n}{2}+1.
\end{equation}
We will repeatedly convert lower bounds for $|H(A)|$ into lower bounds for $N$ through \eqref{eq:interval-upper-H}.

\bigskip
\noindent
For a finite graph $G=(V,E)$, a labeling is a bijection $f:V\to\{1,\ldots,|V|\}$. Its bandwidth is
\begin{equation}\label{eq:bandwidth-definition}
\bw(G)=\min_f\max_{uv\in E}|f(u)-f(v)|.
\end{equation}
Let $Q_n^{(3)}$ denote the Cartesian product of $n$ copies of the three-vertex path; equivalently, its vertex set is $\{0,1,2\}^n$, and two vertices are adjacent when they differ by $1$ in exactly one coordinate.

\bigskip
\noindent
We also record a density-transfer formulation. Let $r_k(X)$ be the maximum size of a $k$-term-progression-free subset of an interval of $X$ consecutive integers, and define
\begin{equation}\label{eq:Rk-definition}
R_k(m)=\min\{X:r_k(X)\ge m\}.
\end{equation}
If $H(A)$ is $k$-term-progression-free and $|H(A)|\ge L$, then \eqref{eq:max-sum} gives
\begin{equation}\label{eq:density-transfer-basic}
R_k(L)\le nN-\binom{n}{2}+1.
\end{equation}
Thus every cube-size lower bound $L$ yields
\begin{equation}\label{eq:density-transfer-g}
g_k(n)\ge \left\lceil\frac{R_k(L)+\binom{n}{2}-1}{n}\right\rceil.
\end{equation}
We will use the elementary choice $R_k(L)\ge L$ for the main theorems, while \eqref{eq:density-transfer-g} allows any quantitative form of Szemer\'edi's theorem to be inserted afterward.

\section{The Three-Term Case}\label{sec:k3}

The key algebraic reduction is exact.

\begin{proposition}\label{prop:ternary-equivalence}
Let $A=\{a_1,\ldots,a_n\}$ be a set of positive integers, and define
\begin{equation}\label{eq:ternary-map}
\Phi_A:\{0,1,2\}^n\longrightarrow\Z,
\qquad
\Phi_A(\eps)=\sum_{i=1}^n\eps_i a_i.
\end{equation}
Then $H(A)$ contains no nonconstant three-term arithmetic progression if and only if $\Phi_A$ is injective.
\end{proposition}

\begin{proof}
Assume first that $\Phi_A$ is injective, and suppose $x,z,y\in H(A)$ satisfy $x+z=2y$. Choose $u,w,v\in\{0,1\}^n$ with $x=\sum_i u_i a_i$, $z=\sum_i w_i a_i$, and $y=\sum_i v_i a_i$. Then
\[
\Phi_A(u+w)=\Phi_A(2v).
\]
Both coefficient vectors belong to $\{0,1,2\}^n$, so injectivity gives $u+w=2v$. Coordinatewise, this forces $u_i=v_i=w_i$ for every $i$, and hence $x=y=z$.

\bigskip
\noindent
Conversely, suppose that $\Phi_A$ is not injective. Choose distinct $c,d\in\{0,1,2\}^n$ with $\Phi_A(c)=\Phi_A(d)$ and set $\delta=c-d$. For each $\delta_i\in\{-2,-1,0,1,2\}$, choose $u_i,v_i,w_i\in\{0,1\}$ so that
\begin{equation}\label{eq:delta-decomposition}
\delta_i=u_i+w_i-2v_i.
\end{equation}
For example, one may use the triples $(u_i,w_i,v_i)=(0,0,1),(1,0,1),(0,0,0),(1,0,0),(1,1,0)$ for $\delta_i=-2,-1,0,1,2$, respectively. Equation \eqref{eq:delta-decomposition} gives subset sums $x=\sum_i u_i a_i$, $z=\sum_i w_i a_i$, and $y=\sum_i v_i a_i$ satisfying $x+z=2y$. If these three numerical sums are not all equal, the equation forces them to be three distinct values in arithmetic progression.

\bigskip
\noindent
It remains to consider the case $x=y=z$. At least two of the binary vectors $u,v,w$ are distinct; call two such vectors $\alpha$ and $\beta$. Put
\[
P=\operatorname{supp}(\alpha)\setminus\operatorname{supp}(\beta),
\qquad
Q=\operatorname{supp}(\beta)\setminus\operatorname{supp}(\alpha).
\]
After cancelling the common support, the equality of the corresponding subset sums gives
\[
\sum_{i\in P}a_i=\sum_{i\in Q}a_i=:s.
\]
Positivity implies that both $P$ and $Q$ are nonempty, so $s>0$. Since $P$ and $Q$ are disjoint, the sums $0,s,2s$ all belong to $H(A)$ and form a nonconstant three-term progression. This proves the converse.
\end{proof}

\begin{corollary}[Integer-linear formulation]\label{cor:integer-linear-formulation}
For every $n\ge1$,
\begin{equation}\label{eq:integer-linear-formulation}
g_3(n)=\min\left\{
\max_{1\le i\le n}a_i:
(a_1,\ldots,a_n)\in\N^n,
\ \eps\mapsto\sum_{i=1}^n\eps_i a_i
\text{ is injective on }\{0,1,2\}^n
\right\}.
\end{equation}
Injectivity in \eqref{eq:integer-linear-formulation} automatically forces the $a_i$ to be pairwise distinct.
\end{corollary}

\begin{proof}
Proposition~\ref{prop:ternary-equivalence} gives the equivalence. If $a_i=a_j$ for $i\ne j$, then the two ternary vectors having a single $1$ in coordinates $i$ and $j$, respectively, have the same image, so injectivity fails.
\end{proof}

\smallskip
\noindent
The next lemma relaxes the integer-linear problem to ordinary graph bandwidth.

\begin{lemma}\label{lem:bandwidth-reduction}
If $A\subseteq[N]$ has $|A|=n$ and $H(A)$ is three-term-progression-free, then
\begin{equation}\label{eq:bandwidth-reduction}
N\ge\bw\!\left(Q_n^{(3)}\right).
\end{equation}
\end{lemma}

\begin{proof}
By Proposition~\ref{prop:ternary-equivalence}, the values of $\Phi_A$ are distinct integers. Label the vertices of $Q_n^{(3)}$ in increasing order of their $\Phi_A$-values. If two vertices are adjacent in coordinate $i$, their numerical values differ by $a_i\le N$. An interval of integral length $a_i$ contains at most $a_i+1$ integers, so the ranks of two distinct integer values at its endpoints differ by at most $a_i$. Every edge therefore has label difference at most $N$, which proves \eqref{eq:bandwidth-reduction}.
\end{proof}

\smallskip
\noindent
The exact bandwidth formula of Billera and Blanco for products of equal paths has the following specialization \cite[Theorem 2.3 and Lemma 3.1]{BilleraBlanco2013}.

\begin{proposition}\label{prop:ternary-grid-bandwidth}
For every $n\ge1$,
\begin{equation}\label{eq:ternary-grid-bandwidth}
\bw\!\left(Q_n^{(3)}\right)=\frac{T_n-1}{2}+\sum_{j=0}^{n-1}T_j.
\end{equation}
\end{proposition}

\begin{proof}
For $j\ge0$, write $C_j(r)=[x^r](1+x+x^2)^j$, so that $T_j=C_j(j)$. Billera and Blanco prove that $\bw(Q_n^{(3)})=\sum_{j=0}^{n-1}\mathcal R_j$, where $\mathcal R_j$ is the sum of the two largest coefficients of $(1+x+x^2)^j$, with $\mathcal R_0=1$. Symmetry and unimodality give $\mathcal R_j=T_j+C_j(j-1)$ for $j\ge1$. Moreover,
\begin{equation}\label{eq:neighbor-trinomial}
T_{j+1}=T_j+2C_j(j-1),
\end{equation}
because the coefficients of $x^{j-1}$ and $x^{j+1}$ in $(1+x+x^2)^j$ are equal. Hence
\[
\sum_{j=0}^{n-1}\mathcal R_j
 =\sum_{j=0}^{n-1}T_j
  +\frac12\sum_{j=1}^{n-1}(T_{j+1}-T_j)
 =\sum_{j=0}^{n-1}T_j+\frac{T_n-1}{2},
\]
which proves \eqref{eq:ternary-grid-bandwidth}.
\end{proof}

\begin{proof}[Proof of Theorem~\ref{thm:main-k3}]
The finite bound \eqref{eq:main-k3-finite} follows immediately from Lemma~\ref{lem:bandwidth-reduction} and Proposition~\ref{prop:ternary-grid-bandwidth}. For the asymptotic statement, the local central limit theorem gives
\begin{equation}\label{eq:trinomial-asymptotic}
T_n\sim\frac{\sqrt3}{2\sqrt{\pi n}}\cdot 3^n.
\end{equation}
For each fixed $m$, this implies $T_{n-m}/T_n\to3^{-m}$. It also implies $T_j/T_{j+1}\to1/3$, so there is $J$ such that
\begin{equation}\label{eq:trinomial-ratio-bound}
\frac{T_j}{T_{j+1}}\le\frac25
\qquad (j\ge J).
\end{equation}
Consequently, for $n-m\ge J$,
\[
\frac{T_{n-m}}{T_n}\le\left(\frac25\right)^m.
\]
The finitely many terms with $n-m<J$ have total $o(T_n)$, while \eqref{eq:trinomial-ratio-bound} gives a summable majorant for the remaining reversed tail. Dominated convergence therefore yields
\begin{equation}\label{eq:trinomial-tail}
\frac1{T_n}\sum_{j=0}^{n-1}T_j
 =\sum_{m=1}^{n}\frac{T_{n-m}}{T_n}
 \longrightarrow\sum_{m=1}^{\infty}3^{-m}=\frac12.
\end{equation}
Combining \eqref{eq:main-k3-finite}, \eqref{eq:trinomial-asymptotic}, and \eqref{eq:trinomial-tail} gives \eqref{eq:main-k3-asymptotic}.
\end{proof}

\begin{remark}[A bandwidth barrier]\label{rem:bandwidth-barrier}
The right-hand side of \eqref{eq:main-k3-finite} is not merely an isoperimetric estimate; it is the exact unrestricted bandwidth of the ternary grid. Therefore no argument that uses only an arbitrary ordering of $\{0,1,2\}^n$ and the fact that adjacent vertices differ in value by at most $N$ can improve Theorem~\ref{thm:main-k3}. Any stronger lower bound must exploit additional arithmetic rigidity of the ordering $\eps\mapsto\sum_i\eps_i a_i$, such as positivity, integrality, or simultaneous behavior of many threshold cuts.
\end{remark}

\begin{remark}[Small values]\label{rem:small-g3-values}
A direct exhaustive enumeration of $n$-subsets of $[N]$, testing the injectivity in Proposition~\ref{prop:ternary-equivalence}, gives the following values for $n\le4$.
\begin{center}
\begin{tabular}{c@{\qquad}c@{\qquad}c@{\qquad}l}
\toprule
$n$ & $g_3(n)$ & $b_n$ & one extremal witness $A$ \\
\midrule
1 & 1  & 1  & $\{1\}$ \\
2 & 3  & 3  & $\{1,3\}$ \\
3 & 8  & 8  & $\{5,7,8\}$ \\
4 & 22 & 21 & $\{7,19,21,22\}$ \\
\bottomrule
\end{tabular}
\end{center}
The first strict gap at $n=4$ illustrates that an optimal unrestricted grid layout need not be induced by a positive integer linear form. The enumeration is small: one forms the $3^n$ values in \eqref{eq:ternary-map} for each candidate set and checks whether they are distinct.
\end{remark}

\smallskip
\noindent
The elementary construction $A=\{1,3,\ldots,3^{n-1}\}$ shows $g_3(n)\le3^{n-1}$. Thus Theorem~\ref{thm:main-k3} leaves a factor of order $\sqrt n$ between the lower and upper bounds. Eliminating this factor would settle the principal question of Erd\H{o}s and S\'ark\H{o}zy.

\section{Adaptive Overlap Growth}\label{sec:adaptive}

We first isolate a one-shift estimate that applies to arbitrary nonzero directions. Afterward we use positivity and pairwise distinctness to average over all unused generators.

\begin{lemma}[Chain overlap]\label{lem:chain-overlap}
Let $B$ be a finite set of integers and let $h\in\Z\setminus\{0\}$. If $B\cup(B+h)$ contains no nonconstant $k$-term arithmetic progression, where $k\ge3$, then
\begin{equation}\label{eq:chain-overlap}
|B\cap(B+h)|\le |B|-\ceil{\frac{|B|}{k-2}}.
\end{equation}
Consequently,
\begin{equation}\label{eq:chain-expansion}
|B\cup(B+h)|
 \ge |B|+\ceil{\frac{|B|}{k-2}}
 \ge\frac{k-1}{k-2}\cdot|B|.
\end{equation}
\end{lemma}

\begin{proof}
After translating the union and replacing $h$ by $-h$ if necessary, we may assume $h>0$. Partition $B$ into maximal $h$-chains, that is, sets of the form $\{x,x+h,\ldots,x+(\ell-1)h\}\subseteq B$ that cannot be extended inside $B$ in either direction. Every chain has length at most $k-2$. Indeed, if $x,x+h,\ldots,x+(k-2)h$ all belonged to $B$, then $x+(k-1)h\in B+h$, producing a $k$-term progression in $B\cup(B+h)$. Therefore the number of chains is at least $\ceil{|B|/(k-2)}$. A chain of length $\ell$ contributes exactly $\ell-1$ elements to $B\cap(B+h)$, so the total overlap is $|B|$ minus the number of chains. This proves \eqref{eq:chain-overlap}, and \eqref{eq:chain-expansion} follows from
\[
|B\cup(B+h)|=2|B|-|B\cap(B+h)|.
\]
\end{proof}

\noindent
Define an integer sequence by
\begin{equation}\label{eq:Fk-recurrence}
F_k(0)=1,
\qquad
F_k(j+1)=F_k(j)+\ceil{\frac{F_k(j)}{k-2}}.
\end{equation}

\begin{corollary}[Universal chain recurrence]\label{cor:chain-only-recurrence}
Let
\[
C=a_0+\{0,h_1\}+\cdots+\{0,h_d\}
\]
be a Hilbert cube contained in a set with no nonconstant $k$-term arithmetic progression. If every $h_j$ is nonzero, with repetitions allowed, then
\begin{equation}\label{eq:chain-only-cube-bound}
|C|\ge F_k(d)\ge d_k^d.
\end{equation}
\end{corollary}

\begin{proof}
Let $C_j=a_0+\{0,h_1\}+\cdots+\{0,h_j\}$. Since $C_j=C_{j-1}\cup(C_{j-1}+h_j)$ is contained in the progression-free set, Lemma~\ref{lem:chain-overlap} gives
\[
|C_j|\ge |C_{j-1}|+\ceil{\frac{|C_{j-1}|}{k-2}}.
\]
Induction yields $|C_j|\ge F_k(j)$. The second inequality in \eqref{eq:chain-only-cube-bound} follows from $F_k(j+1)\ge d_kF_k(j)$.
\end{proof}

\smallskip
\noindent
Now let $A$ be an $n$-element set of distinct positive integers such that $H(A)$ is $k$-term-progression-free. We order the generators adaptively. After selecting $a_1,\ldots,a_i$, put
\begin{equation}\label{eq:partial-cube}
B_i=H(\{a_1,\ldots,a_i\}),
\qquad
s_i=|B_i|,
\end{equation}
with $B_0=\{0\}$ and $s_0=1$. If $h$ is unused, define
\begin{equation}\label{eq:overlap-definition}
r_h(B_i)=|B_i\cap(B_i+h)|.
\end{equation}
Then
\begin{equation}\label{eq:union-overlap}
|B_i\cup(B_i+h)|=2s_i-r_h(B_i).
\end{equation}

\smallskip
\noindent
The distinct positive shifts permit the following averaging estimate.

\begin{lemma}[Average overlap]\label{lem:average-overlap}
Let $B$ be a finite set of integers and let $U$ be a finite set of distinct positive integers. Then
\begin{equation}\label{eq:sum-overlaps}
\sum_{h\in U}|B\cap(B+h)|\le\binom{|B|}{2}.
\end{equation}
In particular, some $h\in U$ satisfies
\begin{equation}\label{eq:average-overlap}
|B\cap(B+h)|
 \le\floor{\frac{|B|(|B|-1)}{2|U|}}.
\end{equation}
\end{lemma}

\begin{proof}
An element of $B\cap(B+h)$ corresponds to a pair $x,x+h\in B$. Each unordered pair of distinct elements of $B$ has a unique positive difference, so it is counted for at most one $h\in U$. Summing over $h$ proves \eqref{eq:sum-overlaps}. Since the overlaps are integers, averaging gives the floor in \eqref{eq:average-overlap}.
\end{proof}

\smallskip
\noindent
Combining Lemmas~\ref{lem:chain-overlap} and \ref{lem:average-overlap} yields an integer-valued recurrence.

\begin{theorem}[Adaptive recurrence]\label{thm:adaptive-recurrence}
Let $k\ge3$, and let $A$ be an $n$-element set of distinct positive integers for which $H(A)$ contains no nonconstant $k$-term arithmetic progression. The elements of $A$ can be ordered so that the quantities $s_i$ in \eqref{eq:partial-cube} satisfy, for $0\le i<n$,
\begin{equation}\label{eq:adaptive-recurrence}
s_{i+1}\ge
2s_i-
\min\left\{
 s_i-\ceil{\frac{s_i}{k-2}},
 \floor{\frac{s_i(s_i-1)}{2(n-i)}}
\right\}.
\end{equation}
\end{theorem}

\begin{proof}
At stage $i$, let $U$ be the set of the $n-i$ unused generators. By Lemma~\ref{lem:average-overlap}, some $h\in U$ satisfies the second overlap bound in \eqref{eq:adaptive-recurrence}. Since $B_i\cup(B_i+h)\subseteq H(A)$, Lemma~\ref{lem:chain-overlap} gives the first overlap bound for every $h\in U$. Choose an $h$ satisfying the average bound as $a_{i+1}$ and apply \eqref{eq:union-overlap}.
\end{proof}

\smallskip
\noindent
For fixed $n$ and $k$, define $L_{k,n}(0)=1$ and, for $0\le i<n$,
\begin{equation}\label{eq:deterministic-recurrence}
L_{k,n}(i+1)=
2L_{k,n}(i)-
\min\left\{
 L_{k,n}(i)-\ceil{\frac{L_{k,n}(i)}{k-2}},
 \floor{\frac{L_{k,n}(i)(L_{k,n}(i)-1)}{2(n-i)}}
\right\}.
\end{equation}

\begin{corollary}[Exact recurrence bound]\label{cor:exact-recurrence-bound}
Under the hypotheses of Theorem~\ref{thm:adaptive-recurrence},
\begin{equation}\label{eq:exact-H-bound}
|H(A)|\ge L_{k,n}(n).
\end{equation}
Consequently,
\begin{equation}\label{eq:exact-g-bound}
g_k(n)\ge
\ceil{\frac{L_{k,n}(n)+\binom{n}{2}-1}{n}}.
\end{equation}
\end{corollary}

\begin{proof}
For $m\ge1$, write the update map in \eqref{eq:deterministic-recurrence} as
\[
\Psi_{k,m}(s)=
\max\left\{
 s+\ceil{\frac{s}{k-2}},
 2s-\floor{\frac{s(s-1)}{2m}}
\right\}.
\]
This map is nondecreasing on the positive integers. The first branch is increasing. For $1\le s\le2m$, the second branch is nondecreasing because
\[
\floor{\frac{s(s+1)}{2m}}-
\floor{\frac{s(s-1)}{2m}}\le2.
\]
For $s\ge2m+1$, its quadratic subtraction is at least $s$, so the first branch dominates; the transition at $s=2m$ is also nondecreasing. Thus Theorem~\ref{thm:adaptive-recurrence} and induction give $s_i\ge L_{k,n}(i)$ for every $i$. Equation \eqref{eq:exact-g-bound} follows from \eqref{eq:interval-upper-H}.
\end{proof}

\begin{lemma}[Logistic growth]\label{lem:logistic-growth}
Let $0\le\ell\le n$, set $M=n-\ell+1$, and suppose $2^\ell\le2M$. Under the hypotheses of Theorem~\ref{thm:adaptive-recurrence}, the first $\ell$ generators can be chosen so that
\begin{equation}\label{eq:logistic-growth}
s_\ell\ge
2M\left[1-\left(1-\frac1{2M}\right)^{2^\ell}\right].
\end{equation}
\end{lemma}

\begin{proof}
For $0\le i<\ell$, there are at least $M$ unused generators. Keeping only the averaging term in \eqref{eq:adaptive-recurrence}, dropping the floor, and using $s_i(s_i-1)\le s_i^2$ gives
\begin{equation}\label{eq:logistic-step}
s_{i+1}\ge2s_i-\frac{s_i^2}{2M}.
\end{equation}
Put $y_i=s_i/(2M)$. Since $s_i\le2^i\le2^\ell\le2M$, we have $0\le y_i\le1$, and \eqref{eq:logistic-step} becomes
\[
y_{i+1}\ge1-(1-y_i)^2.
\]
The map $y\mapsto1-(1-y)^2$ is increasing on $[0,1]$. Starting from $y_0=1/(2M)$ and iterating gives
\[
y_\ell\ge1-\left(1-\frac1{2M}\right)^{2^\ell},
\]
which is \eqref{eq:logistic-growth}.
\end{proof}

\begin{proof}[Proof of Theorem~\ref{thm:main-lower-general}]
Let $\ell=\floor{\log_2n}$ and $M=n-\ell+1$. For every $n\ge2$, one has $2^\ell\le n\le2M$, so Lemma~\ref{lem:logistic-growth} applies. After the first $\ell$ steps, use the chain expansion \eqref{eq:chain-expansion} at each of the remaining $n-\ell$ steps. With $d_k$ as in \eqref{eq:dk-lambda}, this gives
\begin{equation}\label{eq:H-lower-closed}
|H(A)|\ge
2M\left[1-\left(1-\frac1{2M}\right)^{2^\ell}\right]
d_k^{n-\ell}.
\end{equation}
The bracketed factor is bounded below by a positive absolute constant, because $2^\ell/M$ stays between two positive constants. Also $M\asymp n$ and
\begin{equation}\label{eq:dk-ell}
d_k^{-\ell}\asymp_k n^{-\log_2d_k}=n^{-\lambda_k}.
\end{equation}
Hence
\begin{equation}\label{eq:H-lower-asymptotic}
|H(A)|\gg_k d_k^n n^{1-\lambda_k}.
\end{equation}
If $A\subseteq[N]$, then \eqref{eq:interval-upper-H} and \eqref{eq:H-lower-asymptotic} imply
\[
N\gg_k d_k^n n^{-\lambda_k},
\]
which proves \eqref{eq:main-general-lower}. Taking $n$th roots gives \eqref{eq:liminf-general-lower}.
\end{proof}

\smallskip
\noindent
The closed-form argument also gives an explicit finite inequality.

\begin{corollary}[Closed-form finite bound]\label{cor:finite-general-lower}
Let $k\ge3$, let $\ell=\floor{\log_2n}$, and let $M=n-\ell+1$. Then
\begin{equation}\label{eq:finite-general-lower}
g_k(n)\ge
\ceil{
\frac{
2M\left[1-\left(1-\frac1{2M}\right)^{2^\ell}\right]d_k^{n-\ell}
+\binom{n}{2}-1
}{n}
}.
\end{equation}
\end{corollary}

\begin{proof}
Combine \eqref{eq:H-lower-closed} with \eqref{eq:interval-upper-H}.
\end{proof}

\begin{remark}[Density refinement]\label{rem:density-refinement}
Replacing the elementary inequality $R_k(L)\ge L$ by any stronger quantitative estimate in \eqref{eq:density-transfer-g}, with either $L=L_{k,n}(n)$ from \eqref{eq:exact-H-bound} or the closed-form value from \eqref{eq:H-lower-closed}, yields a corresponding refinement. For each fixed $k$, current quantitative forms of Szemer\'edi's theorem alter polynomial or subexponential factors but preserve the exponential base $d_k$.
\end{remark}

\smallskip
\noindent
As $k\to\infty$,
\begin{equation}\label{eq:dk-large-k}
\log d_k=\frac1k+O\!\left(\frac1{k^2}\right),
\qquad
\lambda_k=\frac1{k\log2}+O\!\left(\frac1{k^2}\right).
\end{equation}
Thus Theorem~\ref{thm:main-lower-general} has logarithmic exponential rate $(1+o(1))/k$. Reaching the conjectural scale $\log k/k$ likely requires a mechanism that groups many generators at once rather than controlling one shift at a time.

\section{Distinct-Generator Digit Constructions}\label{sec:upper}

We begin with the progression-free digit language.

\begin{lemma}[Restricted digits]\label{lem:restricted-digits}
Let $p$ be prime, let $k\ge3$, and let $0\le\tau\le\min\{p,k\}-2$. The set of nonnegative integers whose base-$p$ digits all belong to $\{0,1,\ldots,\tau\}$ contains no nonconstant $k$-term arithmetic progression.
\end{lemma}

\begin{proof}
Suppose $x_j=x_0+jd$ for $0\le j<k$ and $d>0$. Let $s=v_p(d)$ and let $u=d/p^s\not\equiv0\pmod p$. Looking at the $s$th base-$p$ digit modulo $p$, the digits of $x_j$ are congruent to $c+ju\pmod p$ for some $c$. If $k\le p$, the first $k$ residues are distinct, while the allowed alphabet has size $\tau+1\le k-1$. If $k>p$, the first $p$ residues cover all of $\F_p$, while the alphabet has size $\tau+1\le p-1$. Both cases are impossible.
\end{proof}

\smallskip
\noindent
The following elementary graph lemma makes the construction finite and explicit.

\begin{lemma}[Nearly-regular graphs]\label{lem:nearly-regular-graph}
Let $D\ge0$ and $r\ge D+1$ be integers. There is a simple graph on $r$ vertices with maximum degree at most $D$ and exactly
\begin{equation}\label{eq:graph-edge-count}
\floor{\frac{Dr}{2}}
\end{equation}
edges.
\end{lemma}

\begin{proof}
The case $D=0$ is trivial. Identify the vertices with $\Z/r\Z$. If $D$ is even, join each vertex to the vertices at cyclic distances $1,\ldots,D/2$; this is a $D$-regular graph. If $D$ is odd and $r$ is even, use cyclic distances $1,\ldots,(D-1)/2$ together with the antipodal perfect matching; this is again $D$-regular. These two cases cover $Dr$ even.

\bigskip
\noindent
If $Dr$ is odd, then both $D$ and $r$ are odd, and $r\ge D+2$. Begin with the $(D-1)$-regular circulant graph using cyclic distances $1,\ldots,(D-1)/2$. The edges of cyclic step $(r-1)/2$ form an odd cycle of length $r$ and are disjoint from the existing edge set. Add a maximum matching from that cycle. This adds $(r-1)/2$ edges, leaves one vertex of degree $D-1$, and gives every other vertex degree $D$. The resulting number of edges is
\[
\frac{(D-1)r}{2}+\frac{r-1}{2}=\frac{Dr-1}{2},
\]
as required.
\end{proof}

\begin{theorem}[Graph construction]\label{thm:graph-construction}
Let $p\ge3$ be prime, let $k\ge3$, and put
\begin{equation}\label{eq:digit-capacity}
q=\min\{p,k\}-1,
\qquad
\tau=q-1=\min\{p,k\}-2.
\end{equation}
For every integer $r\ge q-1$, there is a set $A_r$ of distinct positive integers such that
\begin{equation}\label{eq:graph-size}
|A_r|=\floor{\frac{qr}{2}},
\end{equation}
\begin{equation}\label{eq:graph-height}
\max A_r<2p^{r-1},
\end{equation}
and $H(A_r)$ contains no nonconstant $k$-term arithmetic progression.
\end{theorem}

\begin{proof}
Apply Lemma~\ref{lem:nearly-regular-graph} with $D=q-2=\tau-1$ to obtain a simple graph $G$ on vertex set $\{0,1,\ldots,r-1\}$ having maximum degree at most $\tau-1$ and exactly $\floor{(\tau-1)r/2}$ edges. Define
\begin{equation}\label{eq:graph-generators}
A_r=\{p^i:0\le i<r\}
 \cup\{p^i+p^j:\{i,j\}\in E(G)\}.
\end{equation}
All elements in \eqref{eq:graph-generators} are distinct. Moreover,
\[
|A_r|=r+\floor{\frac{(\tau-1)r}{2}}
 =\floor{\frac{(\tau+1)r}{2}}
 =\floor{\frac{qr}{2}}.
\]
In any subset sum of $A_r$, the coefficient of $p^i$ is at most
\[
1+\deg_G(i)\le\tau<p.
\]
Therefore no carry occurs, and every base-$p$ digit lies in $\{0,1,\ldots,\tau\}$. Lemma~\ref{lem:restricted-digits} shows that $H(A_r)$ is $k$-term-progression-free. Finally, every generator is less than $2p^{r-1}$, proving \eqref{eq:graph-height}.
\end{proof}

\begin{proof}[Proof of Theorem~\ref{thm:main-upper-general}]
Let $q=q_{p,k}$ and $r=\rho_{p,k}(n)$ as in \eqref{eq:q-rho-definition}. Then $r\ge q-1$, so Theorem~\ref{thm:graph-construction} applies, and
\[
|A_r|=\floor{\frac{qr}{2}}\ge n.
\]
Delete excess generators if necessary. Progression-freeness is preserved because the new subset-sum set is contained in the old one. Equation \eqref{eq:graph-height} gives
\[
g_k(n)<2p^{r-1}=2p^{\rho_{p,k}(n)-1},
\]
which proves \eqref{eq:finite-upper-general}. Since
\[
\rho_{p,k}(n)=\frac{2n}{\min\{p,k\}-1}+O_{p,k}(1),
\]
taking $n$th roots and minimizing over primes gives \eqref{eq:limsup-upper-general}.
\end{proof}

\begin{proof}[Proof of Corollary~\ref{cor:large-k-rates}]
The lower estimate follows from \eqref{eq:liminf-general-lower} and the first expansion in \eqref{eq:dk-large-k}. By the prime number theorem, there is a prime $p=(1+o(1))k$. Substituting this prime into \eqref{eq:limsup-upper-general} gives
\[
\log U_k\le\frac{2\log p}{\min\{p,k\}-1}
 =(2+o(1))\cdot\frac{\log k}{k}.
\]
This proves \eqref{eq:large-k-rates}.
\end{proof}

\section{Discussion and Open Problems}\label{sec:discussion}

For three-term progressions, Theorem~\ref{thm:main-k3} and the elementary powers-of-three construction give
\begin{equation}\label{eq:k3-current-gap}
\left(\frac{\sqrt3}{2\sqrt\pi}+o(1)\right)\frac{3^n}{\sqrt n}
 \le g_3(n)\le3^{n-1}.
\end{equation}
The central question is whether the factor $\sqrt n$ can be removed from the lower bound. Remark~\ref{rem:bandwidth-barrier} shows that ordinary ternary-grid bandwidth is exhausted exactly. A successful argument must distinguish integer-linear layouts from arbitrary optimal bandwidth layouts. The strict gap $g_3(4)=22>b_4=21$ in Remark~\ref{rem:small-g3-values} shows that such a distinction already occurs in low dimension. One possible route is a stability theorem showing that a near-optimal layout has many threshold sets close to simplicial initial segments, followed by an arithmetic obstruction to realizing all of those cuts with one positive integer linear form.

\bigskip
\noindent
For general $k$, Corollary~\ref{cor:large-k-rates} leaves a logarithmic gap between the lower and upper exponential rates. The universal chain recurrence already gives the base $(k-1)/(k-2)$ without positivity or distinctness; pairwise distinct positive directions improve only the polynomial factor through overlap averaging. Restricted-digit examples with repeated directions suggest that blocks of approximately $k$ directions may create $k-1$ independent local states and hence a rate of order $\log k/k$. It is therefore natural to ask whether there is an absolute constant $c>0$ such that every $k$-term-progression-free subset-sum set generated by $n$ distinct positive integers satisfies
\begin{equation}\label{eq:target-cube-growth}
|H(A)|\ge\exp\!\left(c\cdot\frac{n\log k}{k}\right).
\end{equation}
Even a proof of \eqref{eq:target-cube-growth} with a small absolute $c$ would attain the correct large-$k$ scale and qualitatively improve Theorem~\ref{thm:main-lower-general}.

\section*{Acknowledgements}
The author used  GPT-5.5 as a research assistant during the preparation of this paper. The author was responsible for the original observation leading to the $3^n/\sqrt n$ asymptotic lower bound in the three-term case. The sharpening of the leading constant was made after an AI-assisted literature review identified the relevant bandwidth formula for the ternary grid. For the general $k$ bounds, the author was responsible for the local one-shift chain idea giving the growth factor $\bigl((k-1)/(k-2)\bigr)^n$. The polynomial refinement obtained by averaging overlaps over the unused distinct generators in the initial steps was suggested by GPT-5.5. The author checked the resulting arguments and is responsible for the final statements, proofs, and any errors.

\end{document}